\documentclass{amsart}

\usepackage{amsmath,amssymb,amsrefs}
\usepackage{tikz-cd}
\usepackage{algorithm}  
\usepackage[noend]{algpseudocode}
\usepackage{Frob-Sch-Sims-MACROS}
\usepackage{url}
\usepackage{hyperref}
\usepackage{todonotes}
\usepackage{bbm}

\definecolor{lightblue}{RGB}{240,240,255}

\usepackage{algorithm,algorithmicx}

\theoremstyle{plain} %% This is the default
\numberwithin{equation}{section}
\newtheorem{thm}[equation]{Theorem}
\newtheorem*{thm*}{Theorem}

\newtheorem{lemma}[equation]{Lemma}
\newtheorem{prop}[equation]{Proposition}

\theoremstyle{definition}

\theoremstyle{remark}
\newtheorem*{remark*}{Remark}
\newtheorem{remark}[equation]{Remark}

\DeclareMathOperator{\Ann}{Ann}

\newcommand{\xrightarrowdbl}[2][]{%
  \xrightarrow[#1]{#2}\mathrel{\mkern-14mu}\rightarrow
}

\title[A Frobenius-Schreier-Sims Algorithm]{A Frobenius-Schreier-Sims Algorithm 
to tensor decompose algebras}
\author{Ian Holm Kessler}
\address{ 
519 S Meldrum St
Apt 313
Fort Collins, CO 80521
} 
\email{ianholm9@gmail.com}
\author{Henry Kvinge}
\address{ 
	Department of Mathematics\\ 
	Colorado State University\\ 
	Fort Collins, CO 80523 
} 
\email{Henry.Kvinge@ColoState.Edu}
\author{James B. Wilson}
\address{
	Department of Mathematics\\
	Colorado State University\\
	Fort Collins, CO 80523
}
\email{James.Wilson@ColoState.Edu}
\thanks{This work was partially supported by NSF grant DMS-1620454 and by
the Hausdorff Institute for Mathematics.}

\begin{document}

\begin{abstract}
We introduce a decomposition of associative algebras into a tensor
product of cyclic modules. This produces a means to encode a basis
with logarithmic information and thus extends the reach of 
calculation with large algebras.  Our technique is an analogue to
the Schreier-Sims algorithm for permutation groups and is a by-product
of Frobenius reciprocity.
\end{abstract}
\maketitle

\section{Introduction}

In 1967 by Charles C. Sims \citelist{\cite{Sims:SS1}} introduced an 
algorithm that given a group $G$ generated by permutations $S$ 
on a finite set $\Omega$ produced a data structure that amongst 
other things could efficiently compute $|G|$, decide if an arbitrary
permutation $\sigma$ was in $G$, and if so write it as a word in 
the original generators.  The algorithm was put to immediate and effective
use in the Classification of Finite Simple Groups.
In the years to follow this algorithm would be improved
by several measures. Some improved worst-case complexity, others made faster
implementations for computer algebra systems such as GAP and Magma, and randomized 
nearly-linear time alternatives were created \citelist{\cite{GAP}\cite{Magma}\cite{Furst-Hopcroft-Luks}
\cite{Knuth}\cite{Jerrum}\cite{Seress}}.  Independently the concept of a base 
(one of the outputs of the algorithm) became a powerful device to 
explore subgroup lattices of large groups 
\citelist{\cite{Bertram-Herzog}\cite{Burness}\cite{Cameron-Solomon-Turull}}.
Today this family of techniques we collectively known as 
\emph{Schreier-Sims} algorithms.  

Here we introduce a Schreier-Sims type algorithm for computing bounds
on the dimension of large algebras, e.g. of group, Hecke, Hopf, and finitely
presented algebras. 

\subsection*{Notation}
We prefer here the notation $g\omega$ for the action of an element $g\in G$ on 
a term $\omega$ in a set $\Omega$ as this will accord well with our use of
left $A$-modules $M$.  We write $S_n$ for the group 
of permutations on $\Omega=\{1,\ldots,n\}$.
A group generated by a set $S$ is denoted $\langle S\rangle$.  

The free $K$-algebra 
on a set of indeterminants $X$ is denoted $K\langle X\rangle$ and consists of all 
$K$-linear combinations of words in $X$.  Also a $K$-algebra generated by a set $S$ 
is denoted $K\langle S\rangle$.   If the elements of $S$ are
known to commute we may also write $K[S]$, e.g. $K[X]$ denotes the usual polynomial
ring in $X$.  Note that in our notation a group algebra
is denoted $K\langle G\rangle$, not as $K[G]$. Let $\Ann_A(M)=\{a\in A\mid aM=0\}$.
Call $M$ \emph{faithful} if $\Ann_A(M)=0$.

\subsection*{An unfaithful use of Schreier-Sims}
Consider a case where an unfaithful representation 
$\rho:A\to \mathbb{M}_n(\mathbb{C})$ of a $\mathbb{C}$-algebra $A$ 
could be used to compute $\dim A$.
Assume first that $G$ is a group with a faithful permutation representation 
$\rho:G\to S_n$ into the symmetric group $S_n$ on $\{1,\ldots,n\}$.  
Take $A=\mathbb{C}\langle G\rangle$ to be the group algebra and
$\hat{\rho}:A\to \mathbb{M}_n(\mathbb{C})$ to be the linear extension
of $G$'s action to permute coordinates of  vectors.  Notice $\hat{\rho}$ 
is usually far from injective.  
E.g. if $n>3$,
\begin{align*}
	\dim \mathbb{C}\langle S_n\rangle=n!\gg n^2=\dim \mathbb{M}_n(\mathbb{C}).
\end{align*}
However, to compute $\dim A$ from these data, we simply apply the above
mentioned Schreier-Sims type algorithms to compute $|G|$ as a permutation group.

Our simple question is whether something like the above example can be
done without a priori knowledge of a group, or other baked-in structure.  
For instance, if all we know is that we have an algebra $A$ and an 
unfaithful $A$-module $M$, can we learn $\dim A$, 
or even a bound on $\dim A$?  On the one hand it seems $M$ is unlikely
to be helpful as it is not faithful.  Yet on the other hand the above
example shows that with the right representation theory perhaps $M$ 
is more informative than expected.  Replicating the success of Schreier-Sims
for general algebras turns up a happy coincidence.  Schreier-Sims is
a special case of Frobenius reciprocity.

Call an algebra $A$ \emph{semiprimary} if its Jacobson radical $J(A)$
is nilpotent and $A/J(A)$ is semisimple.

\begin{thm}\label{thm:main}
Fix a field $K$.  Given a semiprimary $K$-algebra $A$ and an $A$-module $M$, 
there exists $x_1,\ldots,x_{\ell}\in M$, a chain of subalgebras 
$A=A_0>A_1>\cdots >A_{\ell}$, where all irreducible representations 
of $A_{\ell}$ on $M$ are trivial, and a $K$-linear epimorphism
\begin{align*}
	A_0 x_1\otimes\cdots \otimes A_{\ell-1} x_{\ell}\otimes A_{\ell} \twoheadrightarrow A.
\end{align*}
Also given the means to operate in $A$ and the image of a generating set 
of $A$ acting on $M$, there is
a polynomial-time algorithm to construct the above data.
\end{thm}

Amongst the implications of Theorem~\ref{thm:main} is a way to parameterize a spanning set
for $A$ by constructing bases $\mathcal{B}_i$ for each cyclic module $M_i=A_{i-1}x_i$, and
$A_{\ell}$.  The image of $\mathcal{B}_1\otimes\cdots\otimes \mathcal{B}_{\ell}\otimes A_{\ell}$
spans $A$ and provides a generating set for which each monomial term is bounded in length
buy $\ell+1$.  (Note that monomials in the original generators $S$ of $G$
are words of arbitrary length in $K\langle S\rangle$ and thus have numerous unspecified
relations.) Furthermore
we obtain a bound
\begin{align*}
	\dim A & \leq \dim M_1\cdots \dim M_{\ell} \cdot \dim A_{\ell}.
\end{align*}
Each $\dim M_i$ is known, the lingering ambiguity lies with $A_{\ell}$.  While we have
not discovered a property of $A$ and $M$ that will force $\dim A_{\ell}=1$, we have
found in examples it is quite common that $A_{\ell}=K$ or an algebra for which
the dimension is self-evident.

Continuing with our illustration.  Suppose our group $G$ represented on 
$\Omega=\{1,\ldots,n\}$ has a subset $\beta=\{\beta_1,\ldots,\beta_{\ell}\}$ such
that the following stabilizer subgroups
\begin{align}\label{eq:stab}
	G^{[i]} & = \{ g\in G \mid 1\leq j\leq i, \beta_j^g=\beta_j\}
\end{align}
end in $1$, i.e. $G^{[\ell]}=1$.  In the usual Schreier-Sims parlance, $\beta$ is
a \emph{base} for $G$. Use $A_i:=\mathbb{C}\langle G^{[i]}\rangle$
as subalgebras of $A=\mathbb{C}\langle G\rangle$.  Let $x_i=e_{\beta_i}$ be the
vector with zero's in all positions except $\beta_i$ and write $\Delta_i$ for
the $G^{[i-1]}$-orbit of $\beta_i$.  Then for each $i$:
\begin{align*}
	A_{i-1} x_i & = \mathrm{Span}_\mathbb{C}\left\{e_{\delta} \mid 
		\delta\in \Delta_i\right\}.
\end{align*}
Hence, the dimension of $A_{i-1}x_i$ is $|\Delta_i|$.  Finally note that
$A_{\ell+1}=\mathbb{C}$ in this case.  So we obtain a surjection
\begin{align*}
	\mathbb{C}^{\Delta_1}\otimes \cdots \otimes \mathbb{C}^{\Delta_{\ell}}\twoheadrightarrow 
		\mathbb{C}\langle G\rangle.
\end{align*}
Indeed, in this case of we get a bijection.  Observe that all the ingredients
now are in terms of algebras and modules and thus achieve our goal of removing
the permutation interpretation.  

\subsection*{Related work}
Ours is not the first attempt at a generalization of Schreier-Sims for linear
representations.  A notable related work by B\"a\"arnhielm \cite{Baarnhielm}
considered groups $G$ of matrices acting on $\mathbb{K}^n$ for finite fields $\mathbb{K}$
or order $q$. 
There the strategy was to induce from vectors in $\mathbb{K}^n$ a permutation
representation of $G$ and then proceed with the conventional Schreier-Sims 
methods.  The limitations of such an approach are that in general 
linear groups $G$ may fail to have any low-index subgroups.  For instance
the proper subgroups of $\mathrm{SL}_d(q)$ have index at least $q^{cd}$ for some $c$.

Also, while our interest is in an algorithm for so-called ``black-box algebra'',
where we know nothing of the algebra when we begin, it is worth specific mention
that many far superior algorithms for algebras exist when something is known 
about $A$.  Several authors have contributed over decades
to computer algebra systems including GAP \cite{GAP}
and Magma \cite{Magma}, and developed special purpose systems like CHEVIE \cite{CHEVIE}
that calculate with either arbitrary but small algebras, or large 
but prescriptive algebras. 
For instance,  an arbitrary algebra can be  provided by a basis and structure constants
or by a faithful representation.  Algorithms for such algebras where investigated
by Ronyai and later several others.  These are small in our sense because we can
provide them by a basis.  Large algebras in computation include group
algebra, quantum groups (Hopf algebras), Hecke algebras, and finitely presented
algebras.  Algorithms for these algebras can be efficient if they are first told
of additional features, such as the group of a group algebra or appropriate
Lie or Chevalley data.  See the above references for details.

\section{From Schreier-Sims to Frobenius}
In the coarsest explanation of our proof, we point out that
having functions from tensors of submodules into other modules is what one expects
when considering induction-restriction functors -- a natural tool in representation
theories both linear and permutation based.  In retrospect it seems
obvious that when working with stabilizer subgroups we could 
make arguments using induction-restriction with appeals to Frobenius reciprocity.
However, the original Schreier-Sims algorithm was so elegantly explained by
a rewriting formula known as Schreier's lemma that the Frobenius interpretation
never appeared.  Its linear analog therefore had not surfaced either.
We only discovered this relationship by a coincidence observing that the tedious
calculations one is loathed to write when proving Schreier's lemma and
Frobenius reciprocity turn up identical formulas.  This seems the right place 
to begin our proof.

Throughout the many improvements to the original Schreier-Sims algorithm
one key aspect survives intact which is the idea to build generators for
the stabilizer of a point by computing representatives of the cosets of
the stabilizer.  The reason this works, and the reason to attach Schreier's
name to the algorithm, is because of the following observation.

\begin{lemma}[Schreier]\label{lem:Schreier}
Given $G=\langle S\rangle$, $H\leq G$, and a function $\tau:G\to G$,
where $\tau(g)H=gH$, it follows that:
\begin{align*}
	H & = \langle \tau(stH)^{-1} st \mid s\in S, t\in \tau(G)\rangle.
\end{align*}
\end{lemma}

The set $\{\tau(stH)^{-1}st \mid s\in S, t\in \tau(G)\}$ is known a set
of \emph{Schreier generators} for $H$.
Schreier proved this lemma in the context of free groups and the proof
is an early example of rewriting in groups.  Its relevance to permutation
groups is as follows. 
Fix a group $G$ acting on a set $\Omega$.  Note that for $\omega\in G$, 
the function $G\to G\omega$ given by $g\mapsto g\omega$ is constant
on the cosets of the stabilizer
\begin{align*}
	G_{\omega} & = \{ g\in G \mid g\omega=\omega\}.
\end{align*}
Such a function $f$ is said to hide the subgroup $G_{\omega}$.  
Schreier's lemma says that to discover the hidden subgroup we need 
only enumerate the orbit 
\begin{align*}
	G\omega & = \{ g_1\omega,\cdots,g_m\omega\}
\end{align*}
So we define $\tau(g):=g_i$ where $g\omega=g_i\omega$ and apply Schreier's lemma to produce
a set of generators for $G_{\omega}$ of size $|S|\cdot |G\omega|$.
By recursion (together with a careful reduction of the number of generators
as we go) we end up with a the following data.
\begin{description}

\item[Base] a subset $\{\beta_1,\ldots,\beta_{\ell}\}$
of $\Omega$ whose stabilizer chain from \eqref{eq:stab} ends in $1$.

\item[Strong Generators] set $X$ of generators for $G$ with the property
\begin{align*}
	G^{[i]} & = \langle X_i\rangle & X_i:= X\cap G^{[i]}.
\end{align*}
This allow us to treat
the $G^{[i-1]}$ orbit $\Delta_i=\{ g\beta_i\mid g\in G^{[i]}\}$
as a connected Cayley graph 
$\mathrm{Cay}(X_i,\Delta_i) 
		 = \{ (\delta,x\delta) \mid \delta\in \Delta_i, x\in X_i\}$.

\item[Schreier tree] a spanning tree for $\mathrm{Cay}(X_i,\Delta_i)$.

\end{description}
These data are considered the output of the Schreier-Sims algorithm.
For a thurough account we refer the reader to \cite{Seress}.  Our own algorithm
will produce a similar output though regrettably it is far less understood
than the output of Schreier-Sims.

\subsection*{A coincidence with Frobenius reciprocity}
Now let us consider the effect of studying $G_{\omega}$ and $\Delta=\omega^G$
as problem of induction and restriction of permutation representations.

Consider an algebra $B$
contained in an algebra $A$.  To every $A$-module $M$ there is an associated
$B$-module $\Res^A_B(M)=\hom_B(B,M)$ which simply restricts the action 
of $A$ to $B$.  Likewise, every $B$-module $N$ induces an $A$-module
$\Ind^A_B(N)=A\otimes_B N$.  

\begin{thm}[Frobenius Reciprocity]
There is a natural isomorphism
\begin{align*}
	\hom_A(\Ind^A_B(N),M) \cong \hom_B(N,\Res^A_B(N)).
\end{align*}
\end{thm}

As is standard with adjoint pairs we compose them to get endofunctor which
can be applied to modules in a single category.  The one we consider
is the composition $\Ind^A_B\circ \Res^A_B$.  The natural isomorphism
in Frobenius reciprocity then asserts a natural transformation from 
this endofunctor to the identity, a so-called \emph{counit} 
$\epsilon:\Ind^A_B\circ \Res^A_B\to 1$,
i.e.:
\begin{align}\label{eq:iso-1}
\Ind^A_B(\Res^A_B(M))=A\otimes_B \hom_B(B,M)\to M,
\end{align}
Usually it is beneficial to operate with such statements at the level of 
objects in the category, after all these are functors.
However, if one takes care to express this relationship specifically 
we discover our connection to Schreier-Sims.  First observe
that \eqref{eq:iso-1} is realized by the following map: 
for $a \in A$ and $\phi \in \hom_B(B,M)$,
\begin{equation*}
	 a \otimes \phi \mapsto a\cdot \phi(1).
\end{equation*}
Now consider this in the following special case. Fix a faithful 
representation $\rho:G\to S_n$.  Set $A=K\langle G\rangle$ and take $M=K^n$
to be the $A$-module induced by $\rho$.  Next define $B=K\langle H\rangle$, 
where $H=G_\omega$. Unpacking the formulas above we observe
that 
\begin{align*}
	\Ind^A_B(\Res^A_B(M)) & \cong K^{G/H}\otimes_K K^n
\end{align*}
and the action by $A=K\langle G\rangle$ is described as follows.  Fix
a transversal $\tau:G/H\to G$.  For $s\in S$, $tH\in G/H$, and
basis vector $e_i\in K^d$, 
\begin{align*}
	s(tH\otimes e_i) & = st H \otimes e_{\tau(stH)^{-1} st \cdot i} .
\end{align*}
Notice that this expression includes precisely the data in 
Schreier's Lemma~\ref{lem:Schreier}.  For example, if we consider 
induction-restriction of $G$-sets $\Omega$ we find the following
formula.  We need to choose a transversal (which will not alter
the result up to isomorphism) and set:
\begin{align*}
	\Ind^G_H(\Res^G_H( {_G\Omega})) & = {_G G/H}_H \times_{\tau} {_H \Omega}
\end{align*}
where
\begin{align*}
	s(tH,\delta) & = (stH, \tau(stH)^{-1}st \delta).
\end{align*}
We now see that 
the role of Schreier's Lemma in Sims' algorithm was to realize 
the Frobenius counit of the pair $(\Ind^G_H,\Res^G_H)$. 
Fortunately for us, Frobenius
reciprocity holds for much more than permutation modules.

\section{Proof of Theorem~\ref{thm:main}}

Our proof is devised in the following way.  We use a point $x\in M$ to
describe  cyclic module $Ax\leq M$.  This replace the concept of an orbit.
Next we devise subalgebras $B\leq A_x=K+\Ann_A(x)$ to replace the role of point
stabilizers in the original Schreier-Sims algorithm.  
We then want to apply the induction-restriction process to $Ax$ to obtain 
a surjection $A\otimes_{B} \hom_B(B,M)\to M$.  Recovering relations
to reduce the size of the vector space on the left to 
$Ax\otimes {_B M}$ we produce a process similar to Schreier's lemma 
and so create generators for $B$.
Hence, we can efficiently write down a surjection $Ax\otimes B\to A$.  
The final stage is to recursively apply the strategy to $B$.

Notice we have not elected to use $B=A_x$.  Doing so would create a
decomposition as well but one that is vastly larger than $A$.  For
example in the case of a group algebra $A=\mathbb{C}\langle G\rangle$,
the point stabilizer $H=G_{\omega}$ forms a subalgebra $B=\mathbb{C}\langle H\rangle$
of dimension $|H|$ where as the stabilizer subalgebra $A_{e_{\omega}}$ has
dimension $|G|-|H|= |H|(|G:H|-1)$.  Instead what we require of $B$ is
simply that we be able to produce a surjection of $Ax\otimes B\to A$.
In fact the ability to choose many algebras for $B$ has made it unclear
whether one can expect to build a surjection with is also a bijection
and thus obtain a precise dimension for $A$. This would be a helpful
question to resolve.

Our proof is split into three parts.  First we introduce a substitute
for the concept of \emph{Schreier generators} (Proposition~\ref{prop:decomposition}).  
We shall call these \emph{Frobenius-Schreier-Sims (FSS) generators} in part to 
acknowledge the critical role of each person and to distinguish our 
generators from those used in permutation group algorithms.

Second we prove that FFS generators exists for semi-local algebras (Proposition~\ref{prop:existence}).  
Along with that proof we discern a reasonable algorithm to construct FSS 
generators for an algebra (Section~\ref{subsect-repeating-decompositions}).
We use that to confirm  it is possible to compute a decomposition:
\begin{align*}
	A_0 x_1\otimes\cdots \otimes A_{\ell-1} x_{\ell}\otimes A_{\ell}\to A.
\end{align*}

\subsection{Frobenius-Schreier-Sims generators}

Fix a $K$-algebra $A=K\langle S\rangle$ generated by $S$ and a cyclic (left)
$A$-module $M=Ax$. By a \emph{transversal} for $M$ we mean $K$-linear map 
$\tau:M\to A$ such that $\tau(x)=1$ and 
\begin{align*}
	ax & = \tau(ax)x.
\end{align*}
Fix a basis $T$ for $\tau(M)$.  By a \emph{Frobenius-Schreier-Sims section}, relative to
$(S,T,\tau)$, we mean a function $\sigma$ on 
\begin{align*}
	ST=\{st|s\in S, t\in T, \tau(stx)\neq 0\}
\end{align*}
into $A^{\times}$ where
\begin{align*}
	\sigma(st)^{-1} - \tau(stx)\in \mathrm{Ann}_A(x).
\end{align*}
From these data we define the \emph{Frobenius-Schreier-Sims (FSS)} generators as
the following set.
\begin{align*}
	U(S,T,\sigma,\tau) = & \{1\}\\
		& \cup \{ \sigma(st)st : s\in ST\}\\
		 &\cup \{\sigma(st)^{-1}-\tau(stx) :s\in ST\}\\
		& \cup\{st: \tau(stx)=0\}.
\end{align*}

\begin{prop} \label{prop:decomposition}
Under the notation above, for each $a\in A$, there exists $\lambda_i\in K$, 
$t_i\in T$, and $u_{i1},\ldots,u_{i\ell(i)}\in U(S,T,\sigma,\tau)$ such that 
\begin{align*}
	a & = \sum_{i}\lambda_i t_i u_{i1}\cdots u_{i\ell(i)}.
\end{align*}
In particular, $K\langle U\rangle \leq K+\mathrm{Ann}(x)$ and
there is a linear epimorphism $Ax\otimes_K K\langle U\rangle \to A$
given by $ax\otimes b\mapsto \tau(ax)b$.  So each $s\in S$ is 
has tensor rank at most $2$, specifically
\begin{align*}
	sx\otimes \sigma(s)s + x\otimes (s-\tau(sx)\sigma(s)s)\mapsto s.
\end{align*}
\end{prop}
\begin{proof}
Let $s\in S$ and $t\in T$. If $\tau(stx)=0$ then $st\in \mathrm{Ann}_A(x)$.   Likewise,
for $st\in ST$, by definition $\sigma(st)^{-1}-\tau(stx)\in \mathrm{Ann}_A(x)$.
Finally,  as $st x = \tau(st x)x=\tau(stx)x+\alpha(st)x=\sigma(st)^{-1} x$,
it follows that $\sigma(st)st x = x$.  Therefore $\sigma(st)st-1\in \mathrm{Ann}(x)$.
So in all cases $U\subset K+\mathrm{Ann}(x)$.

Assuming $a\in A=K\langle S\rangle$, it follows that there are $\alpha_i\in K$ and a 
sequence of sequences $s_{i1},\ldots,s_{i\ell(i)}\in S$ such that
\begin{align*}
	a & = \sum_i \alpha_i s_{i\ell(i)}\cdots s_{i1}.
\end{align*}
As $1\in T$ and $1\in U(S,T,\sigma,\tau)$, use $u_{i1}=t_{i}=1$ so that:
\begin{align*}
	a & = \sum_i \alpha_i s_{i\ell(i)}\cdots s_{i1}t_{i} u_{i1}.
\end{align*}
Now suppose for induction that for some $j$, there exists $\beta_j\in K$, $t_{i}\in T$, 
and $u_{i(j-1)},\ldots,u_{i1}\in U$ such that
\begin{align}\label{eq:induct}
	a & = \sum_i \beta_i s_{i\ell(i)}\cdots s_{ij}t_{i} u_{i(j-1)}\cdots u_{i1}.
\end{align}
Now we proceed to rewrite each summand.  If $\tau(s_{ij}t_i)=0$ then replace
that term with $t_j u_{ij}$ where $t_{i+1}=1$ and $u_{ij}=s_{ij}t_i\in U$.  
Otherwise, set 
$u_{i(j+1)}=\sigma(s_{ij}t_{i})s_{ij}t_{i}\in U(S,T,\sigma)$.  Note that 
$s_{ij}t_{i}=\tilde{\tau}(s_{ij}t_{i})\sigma(s_{ij}t_i)s_{ij}t_{i}$.  So also set 
$\sum_{k} \lambda_k t_{ik}=\tau(xs_{ij}t_{i})$.  By re-writing we prove:
\begin{align*}
	\beta_i s_{i\ell(i)}\cdots s_{ij}t_{i} u_{i(j-1)}\cdots u_{i1}
	 & = \beta_i s_{i\ell(i)}\cdots s_{i(j+1)}\tilde{\tau}(s_{ij}t_{i}x)
	 	\sigma(s_{ij}t_{i})s_{ij}t_{i} u_{i(j-1)}\cdots u_{i1}\\
	 & = \sum_k \beta_i\lambda_k s_{i\ell(i)}\cdots 
	 	s_{i(j+1)}t_{ik}u_{ij} u_{i(j-1)}\cdots u_{i1}\\
 & \qquad	 	+ \beta_i s_{i\ell(i)}\cdots s_{i(j+1)} 1\alpha(st)u_{ij} u_{i(j-1)}\cdots u_{i1}\\
\end{align*}
In particular, every summand now has been converted into a sum of possibly
several summands each with one fewer $S$ terms, followed by a $T$ term, and $U$ terms
(recalling that $\alpha(st)\in U$ as well).    Therefore re-indexing if necessary
\begin{align*}
	a & = \sum_{m} \gamma_m s_{m\ell(m)}\cdots s_{m(j+1)}t_{m}u_{mj} u_{m(j-1)}\cdots u_{m1}.
\end{align*}
Carrying out the recursion we arrive at
\begin{align*}
	a & = \sum_i \lambda_i t_i u_{i\ell(i)}\cdots u_{i1}.
\end{align*}

Finally let $\Gamma:M\otimes K\langle U\rangle\to A$ be defined on pure-tensors as 
$$m\otimes b \mapsto \tau(m)b.$$
Here we have used the assumption that $\tau$ is linear.
From the decomposition above, given $a\in A$, 
\begin{align*}
	\sum_i \lambda_i (t_i x)\otimes (u_{i\ell(i)}\cdots 
		u_{i1})\mapsto \sum_i \lambda_i t_i u_{i\ell(i)}\cdots u_{i1}=a.
\end{align*}
Therefore $\Gamma$ is surjective.
\end{proof}

\subsection{Existance of FSS generators}
To prove the existence of FSS generators we want to reduce to the case of
central simple rings, i.e. matrices $\mathbb{M}_n(\Delta)$ over division rings
$E$ extending $K$.  As we are afforded a module $M$ for $A$ it is a possible
to begin with a simple submodule.  However, as we cannot assume that $A$
is faithfully represented in on $M$ we need a device to lift our results
to $A$ no matter the presence of a nontrivial annihilator.  The tool we 
choose is to lift the Pierce decomposition by the lifting of idempotents.

So recall that in an associative algebra $A$, a set $e_1,\ldots,e_m$ of
elements in $A$ is a set of pairwise orthogonal idempotents if 
$e_i e_j=\delta_{ij} e_i$.  These idempotents are supplementary if
$1=e_1+\cdots +e_m$.  Idempotents other than $0$ and $1$ are called proper nontrivial.
An idempotent $e$ is primitive if it not the sum of proper nontrivial idempotents.
Finally by a \emph{frame} for $A$ we mean a set of pairwise orthogonal
primitive idempotents that sum to $1$.  For instance, in $\mathbb{M}_n(\Delta)$,
the usual matrix units $E_{ij}=[\delta_{ij}]$ give a natural frame:
$\{E_{11},\ldots, E_{nn}\}$.  Finally we need the following classic lemma
on the lifting of idempotents.

\begin{lemma}
Let $A$ be an algebra an $N$ a nilpotent ideal.  Suppose $e\in A$ such that
$e^2 \equiv e\mod{N}$ and $(e^2-e)^n=0$.  Define
\begin{align*}
	\hat{e} & = e^n\sum_{i=1}^{n-1} \binom{2n-1}{j}e^{n-j-1}(1-e)^j.
\end{align*}
It follows that $\hat{e}^2=\hat{e}$, $e\equiv \hat{e}\mod{N}$, and $\widehat{1-e}=1-\hat{e}$.
In particular in a semiprimary algebra we can lift a frame for $A/J(A)$.
\end{lemma}

\begin{prop}\label{prop:existence}
If $A$ is semiprimary and $M$ is a simple $A$-module, then there exists
a transversal with a Frobenius-Schreier-Sims section.
\end{prop}
\begin{proof}
Let $\rho:A\to \End(M)$ be the induced representation.  As $M$ is simple,
by Jacobson's density theorem the image of $A$ is dense, and in particular
$A/\ker\rho$ is primitive.  By assumption $A$ is semiprimary so $A/J(A)$ is
semisimple Artinian, and so $A/\ker\rho$ is simple.  Therefore we have an
epimorphism $A\to \mathbb{M}_n(\Delta)$ for some division ring $\Delta=\End_{A}(M)$.
By lifting the idempotents $E_{ij}$ to idempotents $e_{ij}\in A$ we can construct
explicit elements in $A$ whose image is a prescribed matrix, i.e.
$\sum_{ij} x_{ij}e_{ij}\mapsto [x_{ij}]$, where $x_{ij}\in e_{1}Ae_{1}$
(here we are using the assumption that the radical has finite length).

Now up to a choice of basis of $M$ as a $\Delta$-vector space, 
each $ax=\sum_{i=1}^n x_i e_i$.  Choose
\begin{align*}
	\tau(ax) & = \sum_{i=1}^n x_{i} e_{i1}.
\end{align*}

For $\sigma$ we proceed as follows.  Since we may assume $\tau(stx)\neq 0$, 
there is some $x_i\neq 0$.  If $x_1\neq 0$, choose
\begin{align*}
	\sigma(st) & = 1-e_{11} + \tau(stx).
\end{align*}
This is invertible with inverse
\begin{align*}
	\sigma(st) & = 1-e_{11} + -x_{11}^{-1}\tau(stx) + x_{11}^{-1} e_{11}.
\end{align*}
Furthermore, $(\sigma(st)^{-1}-\tau(stx))e_1=0$, which satisfies our desired
hypothesis.

If $x_1=0$ then let $x_i$ be the first non-zero.  Apply a permutation by $(1i)$ to
the rows and columns and apply the construction above, then conjugate back.  This
is the value for $\sigma(st)$.
\end{proof}

\subsection{Repeating decompositions}\label{subsect-repeating-decompositions} 
We should now like to consider a recursive application of the above decomposition. 
Evidently $B=K\langle U\rangle$ is a subalgebra so we can treat $M$ as a $B$-module.
Thus we can select a new $y\in M$ and proceed with $By$ in the role of $M$ and $B$ 
in the role of $A$.  The result would be to decompose
$$M_1\otimes\cdots\otimes M_{\ell}\otimes A_{\ell}\to A$$
of cyclic modules $M$.  The only complication is if $B=A$, as then we may find 
ourselves in an infinite corecursion.

So when do we get $A=B$?  Well observe that $B\leq K+\mathrm{Ann}_A(x)$ and so $A=B$
would imply $A=K+\mathrm{Ann}_A(x)$.  Thus we cease our corecursion when
$Ax$ is the trivial module $Kx$.  We are free to choose a cyclic submodule from $M$.
Which means we bottom out once $M$ itself is a product of trivial $A$-modules.
I.e. $M=K^n$ with the action by $A$ simply as scalars.

\subsection{An algorithm}
Now we discuss how to realize this decomposition.  We suppose we have an algebra
of black-box type $A=K\langle S\rangle$ and a representation $\rho:A\to \End(M)$.
What we mean by black-box in this context is that we have algorithms that perform
the operations of the algebras, and the module, and to test equality of elements.
It is typically the last assumption that cause some concern. For example operating 
in a quotient of $K[x_1,\ldots,x_n]$ requires testing when one polynomial is member 
in an ideal. That problem is known to be NP-hard and it is solved in general by 
often difficult Gr\"obner bases methods.  Fortunately once this has been done
the results can be recycled for every subsequent comparison of elements.  To if
this is necessary cost then at least it is a one time cost.

\begin{remark}
As a technical matter this input model is not yet usable in the decision problems
such in the study of P vs. NP since as stated we cannot prove the the operations
satisfy the axioms of an algebra.  Since our algorithms performance can only be
guaranteed under that assumption we should also demand that such algebras be input
along with proofs of the axioms.  That can be done but requires a form of computation
based on type theory and that is a subject for another context; see
\cite{Dietrich-Wilson:type}.  Even so, with our assumptions so far 
our algorithms below should be considered as black-box algorithms in 
the \emph{promise hierarchy}.
\end{remark}

We assume also several now standard algorithms for computing with small rings
and modules, that is ones for which a basis is small enough to produce.  
Detailed accounts of the many methods can be found in \citelist{\cite{HoltEO}\cite{vzG}}.

Here are the steps of our algorithm: given an algebra $A=K\langle S\rangle$ and
the images of $S$ in $\End(M)$, do as follows.

\begin{enumerate}
\item If $M$ is a trivial $A$-module return $A\to A$.

\item Otherwise, use the MeatAxe (or Ronayi's deterministic algorithm) to compute an basis
for an simple submodule $N\leq M$, and fix $0\neq x\in N$; so, $Ax=N$, and also
compute $\Delta=\End_A(N)$ producing a representation $A\to \mathbb{M}_n(\Delta)$.

\item Choose a set of supplementary pairwise orthogonal primitive idempotents $E_{11},\dots,E_{nn}$ 
for $\mathbb{M}_n(\Delta)$
and write them as polynomials in the image of $S$ in $\mathbb{M}_n(\Delta)$, for
instance by expanding $S$ into a basis of the image.

\item Apply the idempotent lifting formula to produce pairwise orthogonal 
primitive idempotents $e_1,\dots,e_n$ in $A$.  Add also $e_0:=1-\sum_i e_i$.

\item Implement the choise of $\tau$ and $\sigma$ above from the Pierce decomposition
given by the idempotents just calculated.

\item Compute the set $U$.  Repeat the process with $A_1:=K\langle U\rangle$.
 and return $Ax\otimes K\langle U\rangle$. Return $Ax_1\otimes\cdots \otimes A_{\ell}x_{\ell}\otimes A_{\ell+1}\to A$.

\end{enumerate}

\section{Examples}

In the master's thesis of the first author \cite{Kessler} an implementation of 
parts of our  generalized Schreier-Sims algorithm were developed.  A particular 
technological adaptation was to explore the algorithm in a parallel functional 
programming paradigm.  A full implementation of our algorithm has not been attempted
but the following examples are included as a demonstration.

\subsection{Example: the dihedral group $D_8$}
To give the reader a sense of how the FSS algorithm operates in a 
classical setting, we apply it to the toy example where $A$ is the 
group algebra of the dihedral group $D_{8}$ over the complex numbers, 
$\MB{C}\langle D_8 \rangle$. 

Recall that $D_8$ can be described by
\begin{equation*}
	D_8 = \{\,r,s \;|\; r^4 = s^2 = 1, \; srs = r^{-1}\}.
\end{equation*}
$D_8$ permutes the points of the square $\{(1,0),(0,1),(-1,0),(0,-1)\} \subset \mathbb{R}^2$ in the usual way. We can associate each of these points to a basis vector $\{e_1, e_2, e_3, e_4\}$ of $\MB{C}^4$. Denote this representation by $W$. Then a $2$-dimensional irreducible submodule of $V$ is generated by the element $e_1 - e_3$. Setting $v_1 = e_1 - e_3$ and $v_2 = e_2 - e_4$, the action of $D_8$ is defined by
\begin{equation} \label{eqn-dihedral-group-realization}
r \mapsto \begin{bmatrix}
    0     & -1  \\
    1       & 0   \\
\end{bmatrix},
\quad\quad
s \mapsto \begin{bmatrix}
    1      & 0  \\
    0       & -1  \\
\end{bmatrix}.
\end{equation}

We choose $\tau: V \rightarrow D_8$ to be given by
\begin{equation*}
\tau(v_1) = 1 \quad\quad \text{and} \quad\quad \tau(v_2) = r
\end{equation*}
and $T = \{1,r\}$. Then 
\begin{equation*}
ST = \{r, s, r^2, sr\}.
\end{equation*}
We observe that 
\begin{align*}
\tau(rv_1) & = r,\\
\tau(sv_1) & = 1,\\
\tau(r^2v_1) & = -1,\\
\tau(srv_1) & = -r.\\
\end{align*}
As the above elements are all invertible in $\MB{C}\langle D_8 \rangle$, we define $\sigma: ST \rightarrow \MB{C}\langle D_8 \rangle^\times$ so that $\sigma(st) = \tau(stv_1)^{-1}$:
\begin{align*}
\sigma(r) & = r^3,\\
\sigma(s) & = 1,\\
\sigma(r^2) & = -1,\\
\sigma(sr) & = -r^3.
\end{align*}
Then since 
\begin{equation*}
\{ \sigma(st)^{-1} - \tau(stv_1) \,|\, st \in ST\} = \{0\}
\end{equation*}
and 
\begin{equation*}
\{ st \,|\, \tau(st) = 0\} = \emptyset,
\end{equation*}
we have that
\begin{equation*}
U(S,T,\sigma,\tau) = \MB{C}\langle r^2,s \rangle \subseteq \MB{C} + \text{Ann}(v_1).
\end{equation*}
Of course $\MB{C}\langle r^2,s \rangle \cong C_2 \times C_2$ (the product of two order $2$ cyclic groups). Thus there is an epimorphism from 
\begin{equation*}
D_8v_1 \otimes \MB{C}\langle C_2 \times C_2 \rangle \rightarrow D_8.
\end{equation*}
But due to dimension considerations, this is in fact an isomorphism.

It can be checked that 
\begin{equation*}
W \cong L \oplus L' \oplus V
\end{equation*}
where $L$ and $L'$ are both 1-dimensional ($L$ is the trivial representation and $L'$ is the representation where $r$ acts as $-1$ and $s$ as $1$).  Thus as described in Section~\ref{subsect-repeating-decompositions} the algorithm stops at this point.

\subsection{Example: degenerate cyclotomic Hecke algebras}

As a demonstration of the generality of our method, we will apply it to a representation of a level three degenerate cyclotomic Hecke algebra $H_n^\lambda$. We begin by describing this algebra and justifying its general interest.

The {\emph{degenerate affine Hecke algebra}} $H_n$ is a generalization of the symmetric group $\Sy{n}$. For simplicity, in this example we will take $H_n$ to be an algebra over $\MB{C}$. $H_n$ is generated by elements $s_1, s_2, \dots s_{n-1}$ and $x_1, x_2, \dots, x_n$, such that $s_1, \dots, s_{n-1}$ satisfy the usual Coxeter generator relations for the symmetric group:
 \begin{align} \label{eqn-sym-grp-relations1}
& s_i^2 = 1,   & 1 \leq i \leq n-1,\\
 &s_is_j = s_js_i, & |i-j| > 1, \;\; 1 \leq i,j \leq n-1,\\
 & s_is_{i+1}s_i = s_{i+1}s_is_{i+1}, & 1 \leq i \leq n-2, \label{eqn-sym-grp-relations3}
 \end{align}
the elements $x_1, \dots, x_n$ commute, and:
\begin{align} \label{eqn-commuting}
& s_jx_i = x_is_j, & i \neq j, j+1,\\
& s_ix_i = x_{i+1}s_i - 1, & 1 \leq i \leq n-1.
\end{align}
As a $\MB{C}$-vector space 
\begin{equation*}
H_n \cong  \MB{C}\langle \Sy{n} \rangle \otimes_{\MB{C}} \MB{C}[x_1,\dots,x_n].
\end{equation*}
Choose some $d$-tuple $\lambda = (\lambda_1, \dots, \lambda_d) \in \MB{Z}^d$. Let $I^\lambda$ be the two-sided ideal of $H_n^\lambda$ generated by the element
\begin{equation} \label{eqn-cyclotomic}
\prod_{i = 1}^d (x_1 - \lambda_i).
\end{equation}
The quotient algebra $H^\lambda_n = H_n/I^\lambda$ is called the {\emph{degenerate cyclotomic Hecke algebra}} associated to $\lambda$. $H^\lambda_n$ is said to be of {\emph{level}} $d$. By abuse of notation we write $x_i, s_i \in H_n^\lambda$ for the images of $x_i, s_i \in H_n$ in this quotient. $H^\lambda_n$ has dimension $\dim(H_n^\lambda) = d^nn!$ \cite[Theorem 3.2.2]{K05}. In particular, as $\MB{C}$-vector spaces,
\begin{equation} \label{eqn-DAHA-decomposition}
H^\lambda_n \cong \MB{C}\langle \Sy{n} \rangle \otimes_{\MB{C}} \MB{C}[ x_1,x_2,x_3 ]
\end{equation}
with $\dim(\MB{C}[ x_1,x_2,x_3 ]) = d^n$.

The claim that degenerate cyclotomic Hecke algebras are generalizations of symmetric groups is justified by the fact that when $\lambda = (0)$, $H^\lambda_n \cong \MB{C}\langle \Sy{n} \rangle$. These algebras have deep connections to Lie theory, for example their representation theory is intimately connected to crystals for quantum groups \cite[Part I]{K05} and their centers are related to parabolic category $\mathcal{O}$ for $\mathfrak{gl}_n(\MB{C})$ \cite{Bru08}. Yet despite this, many aspects of their representation theory are still not fully understood. For example, when $d > 2$, $H_n^\lambda$ is generally not semisimple and the dimensions of simple $H_n^\lambda$-modules are not known. 

Consider a simple $6$-dimensional $H_3^\lambda$-module for $\lambda = (2,2,4)$, which we denote by $V_{2,2,4}$. Explicitly, this representation can be described by:\\
\vspace{2mm}
\begin{equation*}
s_1 \mapsto \begin{bmatrix}
    0      & 1 & 0 & 0 & 0 & 0  \\
    1       & 0 & 0 & 0 & 0 & 0  \\
    0       & 0 & 0 & 0 & 1 & 0  \\
      0       & 0 & 0 & 0 & 0 & 1  \\
        0       & 0 & 1 & 0 & 0 & 0  \\
          0       & 0 & 0 & 1 & 0 & 0  \\
\end{bmatrix},
\quad\quad
s_2 \mapsto \begin{bmatrix}
    0      & 0 & 1 & 0 & 0 & 0  \\
    0       & 0 & 0 & 1 & 0 & 0  \\
    1       & 0 & 0 & 0 & 0 & 0  \\
      0       & 1 & 0 & 0 & 0 & 0  \\
        0       & 0 & 0 & 0 & 0 & 1  \\
          0       & 0 & 0 & 0 & 1 & 0  \\
\end{bmatrix},
\end{equation*}
\vspace{3mm}
\begin{equation*}
x_1 \mapsto \begin{bmatrix}
    2      & -1 & 0 & 0 & 0 & -1  \\
    0       & 2 & 0 & 0 & -1 & 0  \\
    0       & 0 & 2 & -1 & -1 & 0  \\
      0       & 0 & 0 & 2 & 0 & -1  \\
        0       & 0 & 0 & 0 & 4 & 0  \\
          0       & 0 & 0 & 0 & 0 & 4  \\
\end{bmatrix},
\quad\quad
x_2 \mapsto \begin{bmatrix}
    2      & 1 & -1 & 0 & 0 & 0  \\
    0       & 2 & 0 & -1 & 0 & 0  \\
    0       & 0 & 4 & 0 & 1 & 0  \\
      0       & 0 & 0 & 4 & 0 & 1  \\
        0       & 0 & 0 & 0 & 2 & -1  \\
          0       & 0 & 0 & 0 & 0 & 2  \\
\end{bmatrix},
\end{equation*}
\vspace{3mm}
\begin{equation*}
x_3 \mapsto \begin{bmatrix}
    4      & 0 & 1 & 0 & 0 & 1  \\
    0       & 4 & 0 & 1 & 1 & 0  \\
    0       & 0 & 2 & 1 & 0 & 0  \\
      0       & 0 & 0 & 2 & 0 & 0  \\
        0       & 0 & 0 & 0 & 2 & 1  \\
          0       & 0 & 0 & 0 & 0 & 2  \\
\end{bmatrix}.
\end{equation*}

\vspace{3mm}

Denote the basis of $V_{2,2,4}$ chosen via the matrix representation above by: 
\begin{equation*}
\{v_1,v_2,v_3,v_4,v_5, v_6\}. 
\end{equation*}
Since $V_{2,2,4}$ is already assumed to be simple, any non-zero element of $V_{2,2,4}$ will generate $V_{2,2,4}$. We pick $V_{2,2,4} = H_3^\lambda v_1$. It can be checked that $\tau: V_{2,2,4} \rightarrow H_3^\lambda$ can be defined so that
\begin{align*}
&\tau(v_1) = 1,  &&\tau(v_2) = s_1,  \\
&\tau(v_3) = s_2,  &&\tau(v_4) = s_2s_1,  \\
&\tau(v_5) = s_1s_2,  &&\tau(v_6) = s_1s_2s_1.  \\
\end{align*}
Then $\tau(V_{2,2,4})$ is exactly the subalgebra generated by $s_1$ and $s_2$ which is isomorphic to the group algebra of $\Sy{3}$, $\MB{C}\langle \Sy{3}\rangle \subset H_3^\lambda$. Some care must be taken when choosing the elements of $T$. Our goal is that the element obtained after commuting $x_1$, $x_2$, or $x_3$ past $t \in T$ will be invertible in the algebra. In this case, the easiest choice is
\begin{equation*}
T = \{1,s_1,s_2,s_2s_1,s_1s_2,s_1s_2s_1\} \subset \MB{C}\langle \Sy{3} \rangle.
\end{equation*}

In order to define $\sigma: ST \rightarrow (H_3^\lambda)^\times$ we first observe that any element of $st \in ST$ will take one of two forms. Either $s = s_1, s_2$ in which case $st$ is an element of $\Sy{3}$ and we define $\sigma(st) = (st)^{-1}$. Otherwise $s = x_1, x_2, x_3$ and in this case we observe that since $t \in T \subset \MB{C}\langle \Sy{3} \rangle$, using relations \eqref{eqn-commuting} we can rewrite
\begin{equation*}
x_it = a_{t,i}x_j + b_{t,i}
\end{equation*}
where $a_{t,i}, b_{t,i} \in \MB{C} \langle \Sy{3} \rangle$ and $j \in \{1,2,3\}$. Then
\begin{equation*}
\tau(x_itv_1) = \tau(a_{t,i}x_jv_1 + b_{t,i}v_1) = k_j\tau(a_{t,i}v_1) + \tau(b_{t,i}v_1) = k_ja_{t,i} + b_{t,i}
\end{equation*}
where $k_1, k_2 = 2$ and $k_3 = 4$. Thus, provided that $k_ja_{t,i} + b_{t,i}$ is invertible, we can define
\begin{equation*}
\sigma(x_it) = (k_ja_{t,i} + b_{t,i})^{-1}.
\end{equation*}
For the choices of $T$ made above, $k_ja_{t,i} + b_{t,i}$ does happen to be invertible for all $i \in \{1,2,3\}$. 

For such $\tau$ and $\sigma$,
\begin{equation*}
\{ \sigma(st)^{-1} - \tau(stv_1) \;|\; s \in S, T \} = \{0\},
\end{equation*}
and
\begin{equation*}
\{ st : \tau(st) = 0\} = \emptyset.
\end{equation*}
On the other hand
\begin{equation*}
\{1,x_1,x_2,x_3\} \subset \{\sigma(st)st \;|\; s \in S, t\in T \}
\end{equation*}
as we would hope, but the span of $\{\sigma(st)st \;|\; s \in S, t\in T \}$ also contains additional elements that act by $0$ on $v_1$. For example, both $s_1x_1 - 2s_1$, and $s_1x_2 - 2s_1$ are in this set. 

Theorem \ref{thm:main} then says that there is a linear epimorphism 
\begin{equation} \label{eqn-dAHA-decom}
H_3^\lambda v_1 \otimes_{\MB{C}} \MB{C}\langle U \rangle \twoheadrightarrow H_3^\lambda.
\end{equation}
It is clear that $H_3^\lambda v_1 \cong \MB{C} \langle \Sy{3} \rangle$ as vector spaces and we suspect that \eqref{eqn-dAHA-decom} is an alternative decomposition to \eqref{eqn-DAHA-decomposition}. 

What aspects of the structure of $V_{2,2,4}$ does the FSS algorithm recognize? $V_{2,2,4}$ can be realized as follows: let $L(2) \boxtimes L(2) \boxtimes L(4)$ be the 1-dimensional representation of $\MB{C}[x_1,x_2,x_3] \subset H_3$ where $x_1$ and $x_2$ act by multiplication by $2$ and $x_3$ acts by multiplication by $4$. Since $\MB{C}[x_1,x_2,x_3]$ is a subalgebra of $H_3$, then we can construct an induced representation of $H_3$ from $L(2) \boxtimes L(2) \boxtimes L(4)$, 
\begin{equation*}
\ind^{H_3}_{\MB{C}[x_1,x_2,x_3]} L(2) \boxtimes L(2) \boxtimes L(4).
\end{equation*}
It turns out that this representation factors through the quotient $H_3 \xrightarrowdbl[]{p} H_3^\lambda$ and that
\begin{equation*}
V_{2,2,4} \cong p\Big(\ind^{H_3}_{\MB{C}[x_1,x_2,x_3]} L(2) \boxtimes L(2) \boxtimes L(4)\Big).
\end{equation*}
The FSS algorithm cannot see this quotient map, but it does seem to see shadows of it. $H_3^\lambda$ is free as a right $\MB{C}[ x_1,x_2,x_3 ]$-module, with basis $\Sy{3} \subset H^\lambda_3$. The FSS algorithm sees this via the first tensor term on the left hand side of \eqref{eqn-dAHA-decom}. The right tensor term of \eqref{eqn-dAHA-decom} does not appear to be equal to $\MB{C}[ x_1,x_2,x_3 ]$, but as suggested above, probably corresponds to an alternative basis decomposition of $H^\lambda_3$.

Within the theory of degenerate cyclotomic Hecke algebras and for many related generalizations of symmetric groups, the entire representation theory of the tower of algebras can be built up from induced representations from the appropriate subalgebra. In this example at least, the FSS algorithm was able to identify such structure.

\section{Closing remarks}
Our interest here has been in the discovery and proof 
of the decomposition of Theorem~\ref{thm:main}.  We have therefore not sought the most efficient
solution or the broadest applications.  We do however encourage such future development.
There are several questions left unanswered, we offer some we can see ourselves.

Our first and most pressing concern is to capture the behavior of the final
term $A_{\ell}$.  We have found often $A_{\ell}=K$ or context makes its value
immaterial.  But we should like to learn qualities of rings and modules that would
predict when $A_{\ell}=K$.

A second concern is that we do not appear to be able to ``sift'' most elements
in $A$ over the FSS generators.  What we mean by this is that in a traditional
Schreier-Sims algorithm and arbitrary permutation can be written as an product
of the strong generating set by having it act on the individual orbits and
removing the corresponding transversal term.  If at the end the element is nontrivial
we have proved that the permutation was not in the group.  Now if we apply the same
reasoning to our FSS generators we have a clear problem.  Elements $a$ of
$A$ may be in the annihilator of $M$ and so no amount of work acting with $a$
on the points $x_1,\dots,x_{\ell}\in M$ will tell us more than that $a$ annihilates $M$.
That being said, as we observed in the introduction for many elements of $A$ 
not in the annihilator this process allows us to write the element with short
monomials.

Finally, a different rewriting option.  If we have an element expressed as
a polynomial in the non-commuting generators $S$ of $A$ we can use the 
rewriting in our proof of Proposition~\ref{prop:decomposition}
to produce a new expression with fixed monomial lengths.  This step 
does not depend on the faithfulness of our modules. However, because 
we deal with linear combinations and not simply single terms as in 
Schreier's lemma, the number of terms in our some grows perhaps exponentially.
It would be nice to find a rewriting algorithm that can express a word
in FSS form in time polynomial in the final number of terms of the word.

\bibliographystyle{amsplain}

%\bibliography{Frob-Sch-Sims-Refs}

\end{document}